\documentclass[psamsfonts]{amsart}
\usepackage[utf8]{inputenc}
\usepackage{amsfonts}
\usepackage{hyperref}
\usepackage{amsmath}
\usepackage{amssymb}
\usepackage{xcolor}
\usepackage{amsthm}
\usepackage[scr=rsfs]{mathalpha}
\usepackage[all]{xy}
\usepackage{tikz}
\usepackage{tikz-cd}
\usepackage{pdflscape}
\usepackage{pgfplots}

\newtheorem{thm}{Theorem}[section]

\newtheorem{prop}[thm]{Proposition}
\newtheorem{lem}[thm]{Lemma}

\theoremstyle{definition}
\newtheorem{defn}[thm]{Definition}

\newtheorem{exmp}[thm]{Example}

\allowdisplaybreaks[1]
\theoremstyle{remark}
\newtheorem{rem}[thm]{Remark}

\newcommand{\Aut}{{\rm Aut}}

\newcommand{\GL}{{\rm GL}}

\newcommand{\h}{\mathbb{H}}

\makeatletter
\let\c@equation\c@thm
\makeatother
\numberwithin{equation}{section}

\title{$\mathbb{H}_{2n+1}$-structures on odd dimensional projective spaces}
 \author[Cong~Ding]{Cong~Ding} 
 \address[Cong~Ding]
 {School of Mathematical Sciences\\ Shenzhen University\\Guangdong\\ China}
 \email{congding@szu.edu.cn}
 
 \author[Zhijun~Luo]{Zhijun~Luo}
 \address[Zhijun~Luo]
 {Center for complex geometry\\ Institute for Basic Science\\ 55 Expo-ro\\ Yuseong-gu\\ Daejeon\\ 34126\\ Republic of Korea.}
\email{luozj@amss.ac.cn \\ luozj@ibs.re.kr}

\begin{document}
\begin{abstract}
We prove that the Heisenberg group $\h_{2n+1}$ admits infinitely many inequivalent equivariant compactifications into $\mathbb{P}^{2n+1}$ for all $n\geq 1$. This result provides an analog of Hassett-Tschinkel's classical result beyond commutative algebraic groups. 
\end{abstract}

\maketitle
\section{Introduction}
We work over the complex number field $\mathbb{C}$ throughout this paper.

The study of the equivariant compactification of algebraic groups has a long and rich history. For reductive groups, the geometry of their equivariant compactifications is often understood via combinatorial methods (see for example \cite{Knop91} and the references therein). In contrast, for unipotent groups, the simplest example is the vector group $\mathbb{G}_a^n$, for which \cite{AZ22} provides a systematic survey on the study of its equivariant compactifications. In particular, to study the equivariant embeddings of $\mathbb{G}^n_a$ into projective space $\mathbb{P}^n$, a classical result by Hassett and Tschinkel \cite{hassett1999geometry} establishes a bijection between such equivariant embeddings and Artinian local commutative associative unital algebras of dimension $n+1$. This correspondence is now known as the \textit{Hassett-Tschinkel correspondence}. 
Moreover, they demonstrated that such equivariant embeddings are not unique up to equivalence when $n\geq 2$, and that there are infinitely many inequivalent equivariant embeddings when $n\geq 6$. For further studies on equivariant embeddings of $\mathbb{G}_a^n$ into other types of algebraic varieties, we refer the reader to \cite{AS11}\cite{AP14}(for projective hypersurfaces), \cite{Dev15}(for flag varieties), \cite{FH14}(for Fano manifolds of Picard number one), and \cite{AR17}\cite{Dzh22}(for toric varieties).

The Heisenberg group (see Definition \ref{defHeisen}) can be regarded as a unipotent algebraic group that is slightly more intricate than the vector group. As a variety, it is isomorphic to $\mathbb{C}^{2n+1}$; although not commutative as a group, its commutator subgroup is only one-dimensional.
To the best of our knowledge, there has been little work in the literature on the study of its equivariant compactification, except for \cite{Sha2004heisenberg, huan2022campana} from the aspect of arithmetic geometry in dimension $3$. The odd-dimensional projective space $\mathbb{P}^{2n+1}$ is the simplest variety that serves as its equivariant compactification. 

In this short note, we study the equivariant embeddings of the $(2n+1)$-dimensional Heisenberg group (denoted by $\mathbb{H}_{2n+1}$) into $\mathbb{P}^{2n+1}$. Our main result (Theorem \ref{mainthm}) shows that there are infinitely many such equivariant embeddings up to equivalence, \textit{in any dimension}.
Our proof is based on the study of a special class of the equivariant embeddings of $\mathbb{H}_{2n+1}$, in which the central element acts trivially on the boundary, i.e., it fixes every point on the boundary. This condition yields a descended additive action. Furthermore, when the descended action is tautological (Definition \ref{tautolo_des}), we show that there exist infinitely many inequivalent equivariant embeddings of $\mathbb{H}_{2n+1}$ in $\mathbb{P}^{2n+1}$ corresponding to it.

We remark here that the equivariant embeddings of $U_P$ into flag variety $G/P$ has been studied by Cheong \cite{Che17}, where $G$ is a simple complex algebraic group, $P\subset G$ is a parabolic subgroup and $U_P$ denotes the unipotent radical of $P$. According to \cite[Theorem 1]{Che17}, such an embedding of $U_P$ in $G/P$ is unique up to equivalence, except in the cases where $G/P$ corresponds to one of the following: $(A_\ell, \alpha_1), (B_\ell, \alpha_\ell),(C_\ell, \alpha_1)$, or $(G_2, \alpha_1)$. These exceptions arise due to the failure of the Cartan–Fubini type extension theorem of Hwang and Mok \cite{HMCartan}, or because $\Aut(G/P)\neq G$ in those cases. More detailed discussion can be found in \cite{Che17}.
Our result concerns precisely the case when $G/P$ corresponds to $(C_\ell, \alpha_1)$, which is among the exceptional cases in Cheong's uniqueness result.

\section{Preliminaries}
\subsection{Some basic notions}
We recall some definitions here.
\begin{defn}\label{defHeisen}
Let $\mathbb{W}$ be a complex vector space of dimension $2n$, and $\omega$ be a non-degenerate skew-symmetric form on $\mathbb{W}$. The Heisenberg group $\mathbb{H}_{2n+1}$ is the set $\mathbb{W} \times \mathbb{C}$ endowed with the group law:
\[
    (w_1, t_1)\cdot (w_2, t_2) = (w_1+w_2, t_1+t_2+\frac{1}{2}\omega(w_1, w_2)).
\] and corresponding Lie algebra is $\mathbb{W} \times \mathbb{C}$ with Lie bracket given by \[[(w_1, t_1), (w_2, t_2)]=(\mathbf{0}, \omega(w_1,w_2))\]
\end{defn}

It is easy to see that the image of the Lie bracket is the center and is of 1-dimensional. We denote $T:=(\mathbf{0},1)\in \mathbb{H}_{2n+1}$. It is easy to see that $\mathbb{H}_{2n+1}/\mathbb{C}T \cong \mathbb{G}^{2n}_{a}$.

\begin{defn}
    Let $\mathbb{G}$ be a connected linear algebraic group. An irreducible algebraic variety $X$ admits a (left) $\mathbb{G}$-action if there exists a morphism $\psi:\mathbb{G}\times X\to X$, satisfying the standard compatibility conditions. 
    
    A variety $X$ is an equivariant compactification of $\mathbb{G}$ (alternatively we say $X$ admits a $\mathbb{G}$-structure) if the variety admits an \textbf{effective} (left) $\mathbb{G}$-action, namely, the stabilizer of a generic point is trivial and the orbit of a generic point is dense. In other words, we say $X$ is an equivariant embedding of $\mathbb{G}$.
\end{defn}
\begin{defn}\label{actionequi}
    Actions $\alpha_i: \mathbb{G}_i\times X_i \to X_i$ of algebraic groups $\mathbb{G}_i$ on algebraic varieties $X_i$, $i= 1, 2$, are said to be \textit{equivalent} if there are a group isomorphism $\phi: \mathbb{G}_1\to \mathbb{G}_2$ and a variety isomorphism $\psi: X_1 \to X_2$ with the following commutative diagram
    \[	
	\begin{tikzcd}
	\mathbb{G}_1\times X_1 \arrow{d}{\alpha_1} \arrow{r}{(\phi,\psi)} & \mathbb{G}_2\times X_2\arrow{d}{\alpha_2}\\	
	X_1 \arrow{r}{\psi} & X_2.
	\end{tikzcd}
	\]
\end{defn}

\subsection{Classical Hassett-Tschinkel correspondence}
We recall the classical Hassett-Tschinkel correspondence for the additive structure of $\mathbb{P}^n$.
\begin{thm}[\cite{hassett1999geometry}]
 There are one-to-one correspondences between the following objects:
\begin{itemize}
    \item[(a)] additive structures on $\mathbb{P}^n$, i.e. effective actions: $\mathbb{G}_a^n\times \mathbb{P}^n\rightarrow \mathbb{P}^n$ with an open orbit;
   \item[(b)] local commutative associative unital algebras $\mathcal{A}$ of dimension $n+1$.
\end{itemize}
The bijection is considered up to equivalence of actions and algebra isomorphisms.

\end{thm}
The explicit correspondence from (a) to (b) is given as follows. Since $\mathbb{G}^n_a$ acts on $\mathbb{P}^n$ effectively, we can regard $\mathbb{G}^n_a$ as a subgroup of $\operatorname{Aut}(\mathbb{P}^n) = \mathbb{P}\GL_{n+1}(\mathbb{C})$. Denote by $\pi: \GL_{n+1}(\mathbb{C}) \rightarrow \mathbb{P}\GL_{n+1}(\mathbb{C})$ the canonical projection and let $H :=\pi^{-1}(\mathbb{G}^n_a)$. It is easy to see that $H$ is a connected commutative linear algebraic group of dimension $n+1$. 
 Consider $\GL_{n+1}(\mathbb{C})$ as an open subset of $\operatorname{Mat}_{n+1}(\mathbb{C})$ and denote by $\mathcal{A}:=\langle H\rangle$ the associative subalgebra of $\operatorname{Mat}_{n+1}(\mathbb{C})$  generated by $H$. Then $\mathcal{A}$ is the desired commutative unital algebra of dimension $n+1$. 

Conversely from (b) to (a),  if $ \mathcal{A} = \mathbb{C} \oplus \mathfrak{m}$ is local with the unique maximal ideal $\mathfrak{m}$, then its group of invertible elements is $A^{\times} =\mathbb{C}^*(1+\mathfrak{m})$, where $(1+\mathfrak{m}, \times ) \cong  (\mathfrak{m},+) \cong \mathbb{G}_a^n$ via the exponential map. Then the action of $\mathbb{G}^n_a $
is given by the tautological action of $\exp(\mathfrak{m})$ on $\mathcal{A}$.

\subsection{Induced additive structures}
Let $V\cong \mathbb{C}^{2n+2}$.
Since $\mathbb{H}_{2n+1}$ acts on $\mathbb{P}V$ effectively, we can regard it as a subgroup of $\operatorname{Aut}(\mathbb{P}V)=\mathbb{P}\GL_{2n+2}(\mathbb{C})$ and regard its Lie algebra $\mathfrak{h}_{2n+1}$ as a Lie subalgebra of $\operatorname{Mat}_{2n+2}(\mathbb{C})$. Since $\mathbb{H}_{2n+1}$ is unipotent, we may assume that it is contained in $\mathcal{U}_{2n+2}$, the subgroup of upper-triangular matrices with unital diagonal elements. Then $\mathfrak{h}_{2n+1}$ is contained in $\mathfrak{u}_{2n+2}=\operatorname{Lie}(\mathcal{U}_{2n+2})$.

Then we consider the $\mathbb{H}_{2n+1}$-structures on $\mathbb{P}^{2n+1}=\mathbb{P}V$, where $V$ is a $(2n+2)$-dimensional vector space. Let $\mathbb{P}V'$ be the boundary where $V'$ is a $(2n+1)$-dimensional vector subspace in $V$. We focus on a class of special $\mathbb{H}_{2n+1}$-structures on $\mathbb{P}V$, for which the $T$ stabilizes every point on $\mathbb{P}V'$. 

Denote by $o$ the reference point of the open orbit.  Let $\mathbb{I}=\mathbb{C}T$ be the central subgroup of $\mathbb{H}_{2n+1}$ and $v=\overline{\mathbb{I}\cdot o}\backslash \mathbb{I}\cdot o$. Let $\hat{v}$ be a corresponding affine line of $v$ and $\widetilde{V}:=V/\mathbb{C}\hat{v}$. 
We have 
\begin{lem}\label{infinity_unique}
Under the setting that $T$ stabilizes every point on $\mathbb{P}V'$,
    $v$ is independent of the choice of the reference point of the open orbit.
\end{lem}
\begin{proof}
Let the hat denotes an affinization and the bracket denotes the projectivization.
    One can see that $\overline{\mathbb{I}\cdot o}$ is a projective line. Moreover $v=[\frac{d}{ds}((sT) \cdot \hat{o})]$. For anthor $o'$ in the open orbit, we may choose an affinization $\hat{o'}=\hat{o}+\hat{v'}$ for some $v'\in \mathbb{P}V'$. Then $v'=[\frac{d}{ds}((sT) \cdot \hat{o'})]=[\frac{d}{ds}((sT) \cdot \hat{o}+(sT)\cdot \hat{v'})]$. As $\mathbb{I}$ stabilizes $v'\in \mathbb{P}V'$, $\frac{d}{ds}(sT) \cdot \hat{v'}=0$ and hence $v'=[\frac{d}{ds}((sT) \cdot \hat{o'})]=[\frac{d}{ds}((sT) \cdot \hat{o}]=v$.
\end{proof}

\begin{lem}\label{T_act_trivial}
    $\mathbb{I}$ acts trivially on $\mathbb{P}\widetilde{V}$.
\end{lem}
\begin{proof}
    This follows from Lemma \ref{infinity_unique} immediately.
\end{proof}

\begin{prop}\label{ComDia}
   The action of $\mathbb{H}_{2n+1}$ descends to the action of $\mathbb{H}_{2n+1}/\mathbb{I}$ on $\mathbb{P}\widetilde{V}$, which gives a $\mathbb{G}^{2n}_a$-structure on $\mathbb{P}\widetilde{V}$. In terms of commutative diagram, we have 
   \[
    \begin{tikzcd}[row sep=huge]
    \mathbb{H}_{2n+1} \times \mathbb{P}V \arrow[r,"\tau"] \arrow[d,swap,"Q\times pr"] &
    \mathbb{P}V \arrow[d,"pr"] \\
    \mathbb{H}_{2n+1}/\mathbb{I} \times \mathbb{P}\widetilde{V} \arrow[r,"\Theta(\tau)"]  & \mathbb{P}\widetilde{V},
    \end{tikzcd}
\]
where $Q$ is the canonical quotient morphism, $pr$ is the projection of $\mathbb{P}V$ to $\mathbb{P}\widetilde{V}$ and $\Theta(\tau)$ is the induced additive structure of $\tau$.
\end{prop}
\begin{proof}
    For any special $\h_{2n+1}$-structure $\tau$, $T$ stabilizes every point on $V'$.  Let $C = \overline{\mathbb{I} \cdot o}$. For $g\in \h_{2n+1}$,  $g\cdot C$ is still a line in $\mathbb{P}V$. Since $\mathbb{I}$ is the center of $\h_{2n+1}$, we have $$g\cdot v = g\cdot C \backslash g\cdot(\mathbb{I}\cdot o) = \overline{\mathbb{I}\cdot g\cdot o}\backslash (\mathbb{I}\cdot g\cdot o) = v,$$ where the last equality follows from Lemma \ref{infinity_unique}.
    Thus, $v \in \mathbb{P}V$ is a fixed point of the $\mathbb{H}_{2n+1}$-action. Then the action of $\h_{2n+1}$ on $\mathbb{P}V$ can descend to the action on $\mathbb{P}\widetilde{V}$. Moreover by Lemma \ref{T_act_trivial}, we obtain an $(\mathbb{H}_{2n+1}/\mathbb{I})$-action $\Theta(\tau)$ on $\mathbb{P}\widetilde{V}$. Finally,
    since $\tau$ is an equivariant compactification, then so is $\Theta(\tau)$.
\end{proof}

\begin{defn}\label{tautolo_des}
    We call the $\mathbb{G}_a^{2n}$-structure on $\mathbb{P}^{2n}$ coming from the translation the \textbf{tautological $\mathbb{G}_a^{2n}$-structure}, i.e., in homogeneous coordinate the action can be written as  \[(a_1,\cdots, a_{2n})\cdot[z_0,z_1,\cdots, z_{2n}]=[z_0,z_1+a_1z_0,\cdots, z_{2n}+a_{2n}z_0]\]
\end{defn}

Later we are going to show that, there are actually infinitely many inequivalent $\mathbb{H}_{2n+1}$-structures on $\mathbb{P}V$ which descend to the tautological $\mathbb{G}_a^{2n}$-structure on $\mathbb{P}\widetilde{V}$.

\section{Infinity of $\mathbb{H}_{2n+1}$-structures on $\mathbb{P}^{2n+1}$}
In this section,  we still assume that $\mathbb{H}_{2n+1}$ is contained in $\mathcal{U}_{2n+2}$, the subgroup of upper-triangular matrices with unital diagonal elements, and $\mathfrak{h}_{2n+1}$ is contained in $\mathfrak{u}_{2n+2}=\operatorname{Lie}(\mathcal{U}_{2n+2})$.

Denote by $\pi: \GL_{2n+2}(\mathbb{C}) \rightarrow \mathbb{P}\GL_{2n+2}(\mathbb{C})$ the canonical projection and let $\mathcal{A}:=\langle \pi^{-1}(\mathbb{H}_{2n+1})\rangle$ be the associated subalgebra of $\operatorname{Mat}_{2n+2}(\mathbb{C})$ generated by $\pi^{-1}(\mathbb{H}_{2n+1})$. The algebra $\mathcal{A}$ is a local associative unital algebra and the maximal ideal $\mathfrak{m}= \langle \mathfrak{h}_{2n+1}\rangle$ generated by $\mathfrak{h}_{2n+1}$ in $\operatorname{Mat}_{2n+2}(\mathbb{C})$.

For our convenience, we alternatively write the Lie algebra $\mathfrak{h}_{2n+1}=\mathfrak{w}\oplus \mathbb{C}t$ with $\mathbb{C}t$ being the center and $[\mathfrak{w},\mathfrak{w}]=\mathbb{C}t$.
Also we do not distinguish the Lie group $\mathbb{H}_{2n+1}$ (and the Lie algebra $\mathfrak{h}_{2n+1}$) with their image in $\operatorname{Mat}_{2n+2}(\mathbb{C})$.

From the classical Hassett-Tschinkel correspondence for the action of $\mathbb{H}_{2n+1}/\mathbb{C}T$ on $\mathbb{P}\widetilde{V}$, we can write the corresponding unital commutative local algebra as $\mathcal{B}=\mathbb{C}I_{2n+1}\oplus \mathfrak{m}_\mathcal{B} \subset \operatorname{Mat}_{2n+1}(\mathbb{C})$, where $\mathfrak{m}_\mathcal{B}=\mathfrak{h}_{2n+1}/\mathbb{C}t$. The quotient Lie algebra homomorphism $q: \mathfrak{h}_{2n+1} \rightarrow \mathfrak{h}_{2n+1}/\mathbb{C}t$ induces an algebra epimorphism $\theta: \mathcal{A} \rightarrow \mathcal{B}$. 

Additionally, we know that the evaluation morphism $ev: \mathcal{A}\rightarrow \mathcal{A}\cdot \hat{o}$ is a linear map from $\mathcal{A}$ to $V$ where $o$ is the reference point on the open dense orbit in $\mathbb{P}V\cong \mathbb{P}^{2n+1}$ and $\hat{o}$ is the affinization. We can easily observe the following lemmas.
\begin{lem}\label{property_algebra_morphism}
    $\theta(t) = 0$, $\theta$ is injective on $\mathfrak{w}$, $\theta(\mathfrak{w})$ generates $\mathfrak{m}_\mathcal{B}$, and $\theta(\ker(ev)) = 0$.
\end{lem}
\begin{proof}
   The first three statements follows from the construction of $\theta$. For the last statement, for any $S\in \ker(ev)$,  $\theta(S)\cdot pr(\hat{o}) = pr(S\cdot \hat{o})= 0$. Since $\mathfrak{m}_\mathcal{B}=\mathfrak{h}_{2n+1}/\mathbb{C}t$ and the additive action of $\mathbb{H}_{2n+1}/\mathbb{C}T$ is effective, then $\theta(S) = 0$.
\end{proof}

\begin{lem} \label{kernel_ev}
   $\operatorname{ker}(ev) \cap C(\mathcal{A})=0$ where $C(\mathcal{A})$ is the center of $\mathcal{A}$.
\end{lem}
\begin{proof}
    For each $S\in \ker(ev)\cap C(\mathcal{A})$, we have $0=\mathcal{A}S\cdot \hat{o}=S\mathcal{A}\cdot \hat{o}$. Since $\mathcal{A}\cdot \hat{o}$ spans $V$, we have $S=0$.
\end{proof}
\begin{lem}\label{center}
    If $t\cdot \mathfrak{m}=\mathfrak{m}\cdot t=0$, then $\mathfrak{m}^2 \subset C(\mathcal{A})$.
\end{lem}
\begin{proof}
    We can write $\mathcal{A}=\mathbb{C}I_{2n+2}\oplus\mathfrak{m}$. Since $\mathfrak{m}=\langle\mathfrak{h}_{2n+1}\rangle$, it suffices to show that for any $X_1,X_2,\cdots, X_k(k\geq 2), Z\in \mathfrak{h}_{2n+1}$, $\prod_{1\leq i\leq k}X_i\cdot Z=Z\cdot \prod_{1\leq i\leq k}X_i$. This is easy to check. When $k=2$, we have
    \[X_1X_2\cdot Z=X_1\cdot(ZX_2-\lambda t)=X_1\cdot ZX_2=(Z\cdot X_1-\lambda't)X_2=Z\cdot X_1X_2.\]
   General situation can be done by simple induction.
\end{proof}
Then we have
\begin{prop}\label{algebra_dim}
     If $T$ stabilizes the boundary $\mathbb{P}V'$, then $2n+2\leq \dim(\mathcal{A}) \leq 3n+2$. Furthermore, if we additionally assume that the descended $\mathbb{G}^{2n}_a$-action on $\mathbb{P}\widetilde{V}$ is tautological, then $\dim(\mathcal{A})=2n+2$.
\end{prop}
\begin{proof}
  Since $T$ stabilizes the boundary $\mathbb{P}V'$, we may assume that $t=E_{1, 2n+2}$ (the matrix with the unique nonzero entry $1$ at $(1,2n+2)$-position) and $t\cdot \mathfrak{m}=\mathfrak{m}\cdot t=0$. $\dim(\mathcal{A})\geq 2n+2$ is clear since $\pi^{-1}(\mathbb{H}_{2n+1})$ is of dimension $2n+2$.

  It is easy to see that $\ker(ev) \subset \mathfrak{m}$. For nonzero $S_1, S_2 \in \ker(ev)$, $S_1S_2-S_2S_1 \in \ker(ev) \cap \mathfrak{m}^2=\{0\}$, from Lemma \ref{kernel_ev} and Lemma \ref{center}. Hence $\ker(ev)$ is an abelian Lie subalgebra of $\mathfrak{m}$.

  Now, let $S_1,\cdots, S_k$ be the basis of $\operatorname{Ker}(ev)$ and write $S_i=X_i+Y_i (1\leq i\leq k)$ for $X_i\in \mathfrak{w}, Y_i \in \mathfrak{m}^2$ (as $\mathfrak{w}$ and $\mathfrak{m}^2$ span $\mathfrak{m}$). Then $X_1,\cdots,X_k$ are linearly independent in $\mathfrak{w}\subset \mathfrak{h}_{2n+1}$ , since $\ker(ev)\cap \mathfrak{m}^2 =0$. 

  Moreover, by Lemma \ref{center} we know $Y_i\in \mathfrak{m}^2 \subset C(\mathcal{A})$ and then 
    \begin{align*}
        X_i\cdot X_j-X_j\cdot X_i &= (S_i-Y_i)\cdot(S_j-Y_j)-(S_j-Y_j)\cdot(S_i-Y_i)\\
        &=-[S_i,Y_j]-[Y_i, S_j]+[Y_i,Y_j]=0.
    \end{align*}
    
  Therefore, $\langle X_1, \cdots, X_k\rangle\subset \mathfrak{h}_{2n+1}$ is an abelian Lie subalgebra of $\mathfrak{h}_{2n+1}$ contained in $\mathfrak{w}$. This implies $\operatorname{dim}(\ker(ev))=k\leq n$, and $\operatorname{dim}(\mathcal{A})\leq 3n+2$.

  If the descended $\mathbb{G}^{2n}_a$-action on $\mathbb{P}\widetilde{V}$ is tautological, for any $S\in \ker(ev)$, we may write $S = X_S + Y_S$ with $X_S\in \mathfrak{w}$, and $Y_S\in \mathfrak{m}^2$. Note that for the tautological additive action, the corresponding $\mathfrak{m}_\mathcal{B}$ satisfies $\mathfrak{m}^2_\mathcal{B}=0$.
    Therefore, from Lemma \ref{property_algebra_morphism} and  $\theta(Y_S)\in \mathfrak{m}^2_\mathcal{B}$, we have $0=\theta(S) = \theta(X_S)+ \theta(Y_S) = \theta(X_S)$. Again from Lemma \ref{property_algebra_morphism}, together with Lemma \ref{kernel_ev} and Lemma \ref{center}, we have $X_S=0$ and hence $S\in \ker(ev)\cap \mathfrak{m}^2=\{0\}$.  Consequently,
    $\ker(ev)=0$ and $\operatorname{dim}(\mathcal{A}) = 2n+2$.
\end{proof}
\begin{exmp}\label{exm1}
    Let the Heisenberg Lie algebra $\mathfrak{h}_{2n+1}$ be generated by the following $2n+1$ elements:
    \begin{align*}
    & X_i = E_{1,i+1}+E_{i+1,n+i+1}, 1\leq i\leq k\leq n, \\
    & X_j = E_{1,j+1}, k<j\leq n,\\
    & Y_i = E_{1,n+i+1}+E_{n-i+2,2n+2}, 1\leq i\leq n,\\
    & t = E_{1,2n+2},
    \end{align*}
    where $E_{ij}\in \operatorname{Mat}_{2n+2}$ denotes the matrix with a unique nonzero entry $1$ in the $(i, j)$-position.
    Then one can easily check the corresponding algebra $\mathcal{A}$ is dimension $2n+2+k$ for $0\leq k\leq n$.
\end{exmp}

\begin{rem}
One cannot apply the classical Hassett-Tschinkel correspondence of additive structures to obtain the infinity of $\mathbb{H}_{2n+1}$-structures directly, as the induced additive structure does not exhaust all additive structures on $\mathbb{P}^{2n}$, for the following reason.

Since $\mathfrak{w} \cap C(\mathcal{A})=0$ and $\mathfrak{m}^2\subset C(\mathcal{A})$ by Lemma \ref{center}, we have  $\mathfrak{m} = \mathfrak{w} \oplus\mathfrak{m}^2$ and hence $\dim(\mathfrak{m}^2)\leq n+1$ (since $\dim(\mathcal{A})\leq 3n+2$). Then we know $\mathfrak{m}^{n+3} = 0$ by Nakayama Lemma. 
Consider the additive structure on $\mathbb{P}^{2n}$ associated with the commutative local algebra $\mathcal{B}$ such that $\mathfrak{m}_{\mathcal{B}}^{2n}\neq 0$ (e.g, let $\mathcal{B}\cong \mathbb{C}[x]/(x^{2n+1})$), then there does not exist an algebra $\mathcal{A}$ such that $\theta : \mathcal{A} \to \mathcal{B}$ is an epimorphism with $\theta(\mathfrak{m})=\mathfrak{m}_{\mathcal{B}}$ when $n\geq 3$.
\end{rem}

From Proposition \ref{algebra_dim}, we can see that when the descended $\mathbb{G}^{2n}_a$-action is tautological, the unique maximal ideal $\mathfrak{m}$ has dimension $2n+1$ and hence $\mathfrak{m}=\mathfrak{h}_{2n+1}$. 

\begin{defn}
 For an $\mathbb{H}_{2n+1}$-structure on $\mathbb{P}V$.  If $T$ fixes the boundary $\mathbb{P}V'$ and the descended $\mathbb{G}^{2n}_a$-action is tautological, the corresponding unital associative algebra $\mathcal{A}$ is called a \textbf{tautological algebra}.
\end{defn}
For a tautological algebra $\mathcal{A}=\mathbb{C}I_{2n+2}\oplus \mathfrak{m} \subset \operatorname{Mat}_{2n+2}(\mathbb{C})$ with $\mathfrak{m}=\mathfrak{h}_{2n+1}$, let $X_1, \cdots, X_{n}, X_{n+1}, \cdots X_{2n}, t$
be a symplectic basis of $\mathfrak{m}$, i.e.,
$$[X_i,X_j] = X_i\cdot X_{j}-X_{j}\cdot X_i =\omega(X_i,X_j) t,\quad 1 \leq i,j\leq 2n $$ with $$\omega(X_i, X_{n+j}) = \delta_{ij}, \omega(X_i, X_{j}) = \omega(X_{n+i}, X_{n+j}) =0$$ for all $1\leq i,j\leq n$.
Since $\theta(X_i\cdot X_j) =\theta(X_i)\cdot \theta(X_j) \in \mathfrak{m}^2_{\mathcal{B}}=\{0\}$, and $\ker(\theta) =\mathbb{C}t$ (Lemma \ref{property_algebra_morphism}), we obtain that $X_i\cdot X_j = a_{ij}t$.
Therefore, we have a \textbf{structure matrix} $M_{\mathcal{A}} = (a_{ij})\in \operatorname{Mat}_{2n}(\mathbb{C})$ with respect to the symplectic basis, and $M_{\mathcal{A}}-M_{\mathcal{A}}^{t} = \Omega$, the standard skew-symmetric matrix. 

Conversely, for any matrix $M=(a_{ij})$ with $M-M^t = \Omega$, let 
\[
    X_i = E_{1,i+1}+\sum_{l=1}^{2n}a_{li}E_{l+1,2n+2}\in \operatorname{Mat}_{2n+2}(\mathbb{C}), \; 1\leq i\leq 2n,
\]
and $t=E_{1,2n+2}$, then $X_i\cdot X_j = a_{ij}t$, $\mathcal{A} =\mathbb{C}I_{2n+2} \oplus \langle t, X_1,\dots, X_{2n}\rangle$ is a local algebra of dimension $2n+2$, which is a tautological algebra and the unique maximal ideal $\mathfrak{m}$ is isomorphic to $\mathfrak{h}_{2n+1}$ as Lie algebra.
    
\begin{lem}\label{lem1}
    Two tautological algebras $\mathcal{A}_1, \mathcal{A}_2$ are isomorphic if and only if there exists an invertible matrix $C$ such that $M_1 = C\cdot M_2\cdot C^t$ where $M_1, M_2$ are the structure matrices of $\mathcal{A}_1, \mathcal{A}_2$ with respect to some symplectic bases $\{t_1, X_j^1\}$ and $\{t_2, X_j^2\}$ of $\mathfrak{m}_1$ and $\mathfrak{m}_2$ respectively.
\end{lem}
\begin{proof}
    Let $f: \mathcal{A}_1\to \mathcal{A}_2$ be an algebra isomorphism, and let $t_i, X_j^i$ be a basis of $\mathfrak{m}_i$ corresponding to the matrices $M_i$, where $i=1,2$. 
    Since $\mathcal{A}_1, \mathcal{A}_2$ are tautological, $f(\mathfrak{m}_1)=\mathfrak{m}_2$. As $C(\mathfrak{m}_1)=\mathbb{C}t_1, C(\mathfrak{m}_2)=\mathbb{C}t_2$, we have $f(t_1)= kt_2$ for some $k\neq 0$, $f(X_i^1)= \sum c_{ij}X_j^2 + a_i t_2$. Denote $C= (c_{ij})$, since $f(X^1_i\cdot X^1_j) = f(X^1_i)\cdot f(X^1_j)$ and $t_i\cdot \mathfrak{m}_i=\mathfrak{m}_i\cdot t_i=0$, we obtain that $kM_1 = C\cdot M_2\cdot C^t$ and by adjusting $C$ we may assume that $k=1$. The other direction is obvious.
\end{proof}
Two matrices $M_1,M_2$ in $\operatorname{Mat}_{2n}(\mathbb{C})$ are said to be equivalent if there exists an invertible matrix $C$ such that $M_1=C\cdot M_2\cdot C^t$. Then we have
\begin{lem}\label{lem2}
    Let $\mathcal{M}$ be the set of equivalence classes of matrices $M\in\operatorname{Mat}_{2n}(\mathbb{C})$ satisfying $M-M^t = \Omega$, where $\Omega\in \operatorname{GL}_{2n}(\mathbb{C})$, the standard skew-symmetric matrix. Define $\mathcal{S}$ as the set of the orbits of the action of  $\operatorname{Sp}(\Omega)$ on $\operatorname{Sym}_{2n}\subset \operatorname{Mat}_{2n}(\mathbb{C})$, the space of symmetric matrices, given by $(C, M)\mapsto  C\cdot M\cdot C^t$. Then there exists an bijective map between $\mathcal{M}$ and $\mathcal{S}$.
\end{lem}
\begin{proof}
    For every matrix $M$, there exist a unique decomposition $M = \operatorname{S}(M)+\operatorname{AS}(M)$ such that $\operatorname{S}(M):= \frac{1}{2}(M+M^t)$ is symmetric, $\operatorname{AS}(M):= \frac{1}{2}(M-M^t)$ is skew-symmetric. For two equivalent matrices $M_i$ satisfying $M_i-M_i^t = \Omega$ and $M_1=C\cdot M_2 \cdot C^t$, we have $\operatorname{S}(M_1) = C\cdot \operatorname{S}(M_2)\cdot C^t$, and $\Omega = C\cdot\Omega \cdot C^t$. Therefore, we obtain a well-defined map $F:\mathcal{M}\to \mathcal{S}, [M]\mapsto [\operatorname{S}(M)]$. It is clear that $F$ is injective. For any symmetric matrix $N$, then $[N+\frac{1}{2}\Omega]\in \mathcal{M}$, and $F([N+\frac{1}{2}\Omega]) = [N]$. Therefore, $F$ is bijective.
\end{proof}
\begin{prop}\label{infinite_S}
    There are infinitely many elements in $\mathcal{M}$, i.e., there are infinitely may non-isomorphic tautological algebras.
\end{prop}
\begin{proof}
   Since the action of  $\operatorname{Sp}(\Omega)$ on $\operatorname{Sym}_{2n}\subset \operatorname{Mat}_{2n}(\mathbb{C})$ has infinitely many orbits, from Lemma \ref{lem2} we obtain the conclusion.
\end{proof}

Finally we show the following.
\begin{prop}\label{corres}
 If two $\h_{2n+1}$-structures on the projective space $\mathbb{P}V$ are equivalent, the corresponding associative algebras are isomorphic.
\end{prop}
\begin{proof}
    Let $\alpha_i: \h^i_{2n+1} \to \mathbb{P}\operatorname{GL}(V)$ be two equivalent $\h_{2n+1}$-structures on the projective space $\mathbb{P}V$. That is, there are isomorphisms $\phi:\h^1_{2n+1}\to \h^2_{2n+1}$ and $\psi\in \operatorname{GL}(V)$ such that $\alpha_2\circ \phi = \psi\circ \alpha_1 \circ \psi^{-1}$. Consider the corresponding Lie algebra isomorphism  $d\phi:\mathfrak{h}^1_{2n+1}\to \mathfrak{h}^2_{2n+1}$ and its extension on the universal enveloping algebra $\Phi: \mathbb{U}(\mathfrak{h}^1_{2n+1})\to\mathbb{U}(\mathfrak{h}^2_{2n+1})$, we obtain the following commutative diagram, where $\tau_1, \tau_2$ are induced by Lie algebra monomorphism $d\alpha_1, d\alpha_2$ respectively, and  $\operatorname{Ad}_\psi$ denotes the conjugation by $\psi \in \operatorname{GL}(V)$.
    \begin{equation*}
        \begin{tikzcd}[row sep=huge]
        \mathbb{U}(\mathfrak{h}^1_{2n+1}) \arrow[r,"\tau_1"] \arrow[d,swap,"\Phi"] &
        \operatorname{End}(V) \arrow[d,"\operatorname{Ad}_{\psi}"] \\
        \mathbb{U}(\mathfrak{h}^2_{2n+1}) \arrow[r,"\tau_2"]  & \operatorname{End}(V),
        \end{tikzcd}
    \end{equation*}
    then the induced homomorphism $\tilde{\phi}:\mathcal{A}_1\to \mathcal{A}_2$ is an isomorphism with $\mathcal{A}_i = \tau_i(\mathbb{U}(\mathfrak{h}^i_{2n+1}))(i=1,2)$ and the restriction on $\mathfrak{h}^1_{2n+1}$ is the Lie algebra isomorphism. 
 
\end{proof}

\begin{thm}\label{mainthm}
    There are infinitely many inequivalent $\mathbb{H}_{2n+1}$-structures on $\mathbb{P}V = \mathbb{P}^{2n+1}$ such that the induced $\mathbb{G}^{2n}_a$ structures on $\mathbb{P}\widetilde{V}$ are tautological. In particular, there are infinitely many inequivalent $\mathbb{H}_{2n+1}$-structures on $\mathbb{P}V = \mathbb{P}^{2n+1}$.
\end{thm}
\begin{proof}

 Applying Proposition \ref{corres} on tautological algebras and from Proposition \ref{infinite_S}, we obtain the conclusion. 
\end{proof}
 
\section*{Acknowledgement}
	 The authors would like to thank Jun-Muk Hwang for introducing the problem and some helpful discussions. The first author would like to thank Zhizhong Huang for providing references. The first author was supported by a start-up funding from Shenzhen University (000001032064) and Shenzhen Peacock Plan. The second author was supported by the Institute for Basic Science (IBS-R032-D1-2025-a00).
\bibliographystyle{alpha}
\bibliography{ref}
\end{document}